\documentclass[12pt]{article}
\input{amssym}

\newtheorem{prethm}{{\bf Theorem}}

\newenvironment{thm}{\begin{prethm}{\hspace{-0.5
em}{\bf.}}}{\end{prethm}}
\newtheorem{precor}{{\bf Corollary}}

\newenvironment{cor}{\begin{precor}{\hspace{-0.5
em}{\bf.}}}{\end{precor}}
\newtheorem{preprop}{{\bf Proposition}}

\newtheorem{preque}{{\bf Question}}

\newtheorem{preques}{{\bf Question}}

\newtheorem{prelemma}{{\bf Lemma}}

\newtheorem{prefact}{{\bf Fact}}

\newtheorem{preobs}{{\bf Observation}}

\newtheorem{prefig}{{\bf Figure}}

\newtheorem{prelemm}{{\bf Lemma}}

\newtheorem{preex}{{\bf Example}}

\newtheorem{prepro}{{\bf Proposition}}

\newenvironment{pro}{\begin{prepro}{\hspace{-0.5
em}{\bf.}}}{\end{prepro}}
\newtheorem{prelem}{{\bf Theorem}}

\newenvironment{lem}{\begin{prelem}{\hspace{-0.5
em}{\bf.}}}{\end{prelem}}
\newtheorem{preproof}{{\bf Proof.}}

\newenvironment{proof}[1]{\begin{preproof}{\rm
               #1}\hfill{$\rule{2mm}{2mm}$}}{\end{preproof}}

\newtheorem{preconj}{{\bf Conjecture}}

\newtheorem{predeff}{{\bf Definition}}

\newenvironment{deff}{\begin{predeff}{\hspace{-0.5
em}{\bf.}}}{\end{predeff}}
\def\newpic#1{}
\date{}
\begin{document}

\title{
{\Large{\bf Extremal Graph Theory for Metric Dimension and Girth}}}
%

{\small
\author{
{\sc Mohsen Jannesari} 
\\
}


\maketitle \baselineskip15truept
\begin{abstract}
 A set $W\subseteq V(G)$ is called a  resolving set for $G$, if
for each two distinct vertices $u,v\in V(G)$ there exists $w\in W$
such that $d(u,w)\neq d(v,w)$,  where $d(x,y)$ is the distance
between the vertices $x$ and $y$. The minimum cardinality of a
resolving set for $G$ is called the  metric dimension of $G$, and
denoted by $\beta(G)$. In  this paper, it is  proved that in a
connected graph $G$ of order $n$ which has a cycle, $\beta(G)\leq n-g(G)+2$,
where $g(G)$ is the length of a shortest cycle in $G$, and the
equality holds if and only if $G$ is a cycle, a complete graph or a
complete bipartite graph $K_{s,t}$, $ s,t\geq 2$.
\end{abstract}

\section{Introduction}
 Throughout the paper, $G$ is a
finite, simple, and connected  graph of order $n$ with vertex set
$V(G)$ and edge set $E(G)$. The distance
between two vertices $u$ and $v$, denoted by $d(u,v)$, is the
length of a shortest path between $u$ and $v$ in $G$. The diameter
of $G$, denoted by $diam(G)$ is $\max\{d(u,v)\ |\ u,v\in V\}$. The
degree of a vertex $v$, $\deg(v)$, is the number of
its neighbors.
The notations $u\sim v$
and $u\nsim v$ denote the adjacency and non-adjacency relations
between $u$ and $v$, respectively. A cycle of order $n$, $C_n$, is denoted by $(v_1,v_2,\ldots,v_n,v_1)$. The number of edges in a cycle is its
length. If $G$ has a cycle, then  the length of a shortest cycle in $G$
is called the girth
of $G$ and denoted by $g(G)$.
  For a subset $S$ of $V(G)$,
$G\setminus S$ is the induced subgraph $\langle V(G)\setminus
S\rangle$ by $V(G)\setminus S$ of $G$.
A vertex $v\in V(G)$ is a {\it cut vertex} in
$G$ if $G\setminus\{v\}$ has at least two components. If $G\neq K_n$ has no
cut vertex, then $G$ is called a {\it$2$-connected} graph.
\par
For an ordered set $W=\{w_1,w_2,\ldots,w_k\}\subseteq V(G)$ and a
vertex $v$ of $G$, the  $k$-vector
$$r(v|W):=(d(v,w_1),d(v,w_2),\ldots,d(v,w_k))$$
is called  the {\it metric representation}  of $v$ with respect
to $W$. The set $W$ is called a {\it resolving set} for $G$ if
distinct vertices have different representations.  A resolving
set for $G$ with minimum cardinality is  called a {\it metric
basis}, and its cardinality is the {\it metric dimension} of $G$,
denoted by $\beta(G)$.
It is obvious that to see  whether a given set $W$ is a resolving
set, it is sufficient to consider the
 vertices in $V(G)\backslash W$, because $w\in
W$ is the unique vertex of $G$ for which $d(w,w)=0$.
\par In~\cite{Slater1975}, Slater introduced the idea of a resolving
set and used a {\it locating set} and the {\it location number}
for a resolving set and the metric dimension,
respectively. He described the usefulness of these concepts when
working with U.S. Sonar and Coast Guard Loran stations.
Independently, Harary and Melter~\cite{Harary} discovered the
concept of the location number as well and called it the metric
dimension. For more results related to these concepts
see~\cite{k-dimensional,cartesian product,bounds,sur1,On randomly,landmarks}.
The
concept of a resolving set has various applications in diverse
areas including coin weighing problems~\cite{coin}, network
discovery and verification~\cite{net2}, robot
navigation~\cite{landmarks}, mastermind game~\cite{cartesian
product}, problems of pattern recognition and image
processing~\cite{digital}, and combinatorial search and
optimization~\cite{coin}.
Chartrand et al.~\cite{Ollerman} obtained the following bound
 for metric dimension in terms of order and diameter.
\begin{lem}~{\rm\cite{Ollerman}}\label{thm:B<n-d}
If $G$ is a connected graph of order $n$, then $\beta(G)\leq
n-diam(G)$.
\end{lem}
Ten years later, Hernando et al.~\cite{extermal} characterized all graphs
$G$ of order $n$ and metric dimension $n-diam(G)$. The main goal of this paper is to
prove that for a connected graph $G$ of order $n$ and girth $g(G)$
$$\beta(G)\leq n-g(G)+2$$
 and characterize all graphs  such that
 this bound is tight for them. In fact, it is proved that cycles, complete
and complete bipartite graphs are all graphs with $\beta(G)=n-g(G)+2$.
To prove the main results  the following known results are needed.
It is obvious that for a graph $G$ of order $n$,
$1\leq\beta(G)\leq n-1$. Chartrand et al.~\cite{Ollerman}
characterized all graphs of order $n$ and metric dimension $n-1$.
\begin{lem}~{\rm\cite{Ollerman}}\label{thm:B=n-1}
Let $G$ be a graph of order $n$. Then $\beta(G)=n-1$ if
and only if $G=K_n$.
\end{lem}
They also characterized all graphs of order $n$ and metric dimension $n-2$.
\begin{lem} {\rm\cite{Ollerman}} \label{thm:n-2}
Let $G$ be a  graph of order $n\geq 4$. Then $\beta(G)=n-2$ if and
only if $G=K_{s,t}~(s,t\geq 1), G=K_s\vee\overline K_t~(s\geq 1,
t\geq 2)$, or $G=K_s\vee (K_t\cup K_1)~ (s,t\geq 1)$.
\end{lem}
\par
The following definition is needed to state some results in the next section.
\begin{deff}\label{def:ear}
An \emph{ear} of a graph $G$ is a maximal path whose internal vertices are
of degree $2$ in $G$. An \emph{ear decomposition} of $G$ is a decomposition
$G_0,G_1,\ldots,G_k$ such that $G_0$ is a cycle and for each $i,~1\leq i\leq k$,
$G_i$ is an ear of $G_0\cup G_1\cup\ldots\cup G_i$.
\end{deff}
Whitney \cite{whitney} proved the following important result for $2$-connected graphs.
\begin{lem}\label{thm:whitney}~\rm\cite{whitney}
A graph is $2$-connected if and only if it has an ear decomposition. Moreover,
every cycle in a $2$-connected graph is the initial cycle in some ear decomposition.
\end{lem}
\section{Main Results}\label{main}
The aim of this section is to find a bound for metric dimension in terms of
order and girth of a graph and characterize all graphs which attend this bound.
This bound is presented in the next theorem.
\begin{thm}\label{thm:B<n-g+2}
Let $G$ be a graph of order $n$. If $G$ has a cycle, then $\beta(G)\leq n-g(G)+2$.
\end{thm}
\begin{proof}{
Let $g(G)=g$ and $C_g=(v_1,v_2,\ldots,v_g,v_1)$ be a shortest cycle in $G$.
Since  $\{v_1,v_2\}$ is a metric basis of $C_g$, $V(G)\setminus\{v_3,\ldots,v_g\}$
is a resolving set for $G$ of size $n-g(G)+2$. Therefore, $\beta(G)\leq n-g(G)+2$.
}\end{proof}
Note that,  $\beta(K_n)=n-1=n-g(K_n)+2$, $\beta(C_n)=2=n-g(C_n)+2$, and for $r,s\geq2$,
$\beta(K_{r,s})=r+s-2=n-g(K_{r,s})+2$. Therefore, the bound in Theorem~\ref{thm:B<n-g+2}
is tight for these graphs. In the remainder of this section, it is proved that these are
all graphs such that this bound is tight for them. First some  required results are
presented.
\begin{pro}\label{pro:components}
Let $v$ be a cut vertex in a graph $G$. Then each resolving set for $G$
can be disjoint from at most one component of $G\setminus\{v\}$. Moreover, if
$W$ is a resolving set for $G$ which is not disjoint from two
components of $G\setminus\{v\}$, then $W\setminus\{v\}$ is a resolving set for $G$.
\end{pro}
\begin{proof}{
Let $W$ be a resolving set for $G$ and $H$ and $K$ be two
components of $G\setminus\{v\}$. If $W\cap V(H)=W\cap V(K)=\emptyset$, then let
$x\in V(H)$ and $y\in V(K)$ such that $x\sim v$ and $y\sim v$. Therefore, for
each $w\in W$,
$$d(x,w)=d(x,v)+d(v,w)=1+d(v,w)=d(y,v)+d(v,w)=d(y,w),$$
which contradicts the assumption that $W$ is a resolving set for $G$. Thus $W$
can be disjoint from at most one component of $G\setminus\{v\}$.
\par
Now let $W\cap V(H)\neq\emptyset$, $W\cap V(K)\neq\emptyset$, $x\in W\cap V(H)$,
and $y\in W\cap V(K)$. If $W\setminus\{v\}$ is not a resolving set for $G$,
then there exist vertices $a,b\in V(G)$ such that $d(a,v)\neq d(b,v)$ and
for each $w\in W\setminus\{v\}$, $d(a,w)=d(b,w)$. If $a,b\notin V(H)$, then
$$d(a,x)=d(a,v)+d(v,x)\neq d(b,v)+d(v,x)=d(b,x).$$ This gives $d(a,x)\neq d(b,x)$,
which is impossible. Hence, $a,b\in V(H)$. Therefore,
$$d(a,y)=d(a,v)+d(v,y)\neq d(b,v)+d(v,y)=d(b,y).$$ That is $d(a,y)\neq d(b,y)$.
This contradiction implies that $W\setminus\{v\}$ is a resolving set for $G$.}\end{proof}
\begin{cor}\label{cor:leaf}
Let $u$ be a vertex of degree $1$ in a graph $G$ and $v$ be the neighbor of $u$.
If $W$ is a resolving set for $G$,
then $(W\cup\{u\})\setminus\{v\}$ is also a resolving set for $G$.
\end{cor}
\begin{proof}{
Let $W$ be a resolving set for $G$. Clearly $W\cup\{u\}$ is also
a resolving set for $G$. Note that $v$ is a cut vertex of $G$
and $\langle\{u\}\rangle$ is a component of $G\setminus\{v\}$.
If $W\cap(V(G)\setminus\{u,v\})\neq\emptyset$, then by
Proposition~\ref{pro:components},
$(W\cup\{u\})\setminus\{v\}$ is also a resolving set for $G$.
If $W\subseteq\{u,v\}$, then  by
Proposition~\ref{pro:components}, $G\setminus\{v\}$ has exactly two components.
On the other hand, for each $x\in V(G)\setminus\{u\}$, $r(x|\{u,v\})=(a,a-1)$,
for some integer $a\geq1$. Since $\{u,v\}$ is a resolving set for $G$, the first
coordinate of the metric representation of all vertices in $V(G)\setminus\{u\}$
are different from each other. Therefore,
$\{u\}=(W\cup\{u\})\setminus\{v\}$ is also
 a resolving set for $G$.
}\end{proof}
The following proposition states that all graphs $G$ of order $n$ with
metric dimension $n-g(G)+2$ is $2$-connected.
\begin{pro}\label{pro:2-connected}
Let $G$ be a graph of order $n$ which has a cycle. If $\beta(G)=n-g(G)+2$,
then $G$ is $2$-connected.
\end{pro}
\begin{proof}{ Suppose, on the contrary, that $v$ is a cut vertex of $G$.
Let $C_g=(v_1,v_2,\ldots,v_g,v_1)$ be a shortest cycle in $G$. Since $v$
is a cut vertex, there exists a component $H$ of $G\setminus\{V\}$, such that
$V(C_g)\subseteq V(H)\cup\{v\}$. By the assumption, $V(G)\setminus\{v_3,\ldots,v_g\}$
is a basis of $G$, which is not disjoint from at least two
components of $G\setminus\{v\}$.
Therefore, by Proposition~\ref{pro:components},
$v\in\{v_3,\ldots,v_g\}$, say $v=v_i$, $3\leq i\leq g$. But
$B=V(G)\setminus\{v_1,\ldots,v_{i-2},v_{i+2},\ldots,v_g\}$ is a basis of $G$ which contains
$v=v_i$, since $B$ is not disjoint from at least two
components of $G\setminus\{v\}$.
Thus, by Proposition~\ref{pro:components}, $B\setminus\{v\}$ is a resolving
set for $G$ of size smaller than $\beta(G)$. This contradiction implies that $G$ is
a $2$-connected graph.
}\end{proof}
\begin{thm}\label{thm:characterisation}
Let $G$ be a graph of order $n$ which has a cycle. Then $\beta(G)=n-g(G)+2$
if and only if $G$ is a cycle $C_n$, complete graph $K_n$, $n\geq3$, or
complete bipartite graph,
$K_{r,s}$, $r,s\geq2$.
\end{thm}
\begin{proof}{It is easy to see that if $G$ is a cycle $C_n$,
complete graph $K_n$, $n\geq3$, or
complete bipartite graph,
$K_{r,s}$, $r,s\geq2$, then   $\beta(G)=n-g(G)+2$.
Now let $\beta(G)=n-g(G)+2$ and $C_g=(v_1,v_2,\ldots,v_g,v_1)$ be a shortest
cycle in $G$. By Proposition~\ref{pro:2-connected}, $G$ is a $2$-connected graph.
Therefore, By Theorem~\ref{thm:whitney}, $G$ has an ear decomposition with initial
cycle $C_g$. Assume that $C_g,G_1,G_2,\ldots,G_k$ be an ear decomposition of
$G$ with initial cycle $C_g$. If $G=C_g$, then $G$ is a cycle as it is desired. Now let
$G\neq C_g$ and
$G_1=(x_0,x_1,\ldots,x_t)$. Thus
$x_0,x_t\in V(C_g)$. Without loss of generality one can assume that
 $x_0=v_i$ and $x_t=v_j$, where $1\leq i<j\leq g$.
Since $C_g$ is a shortest cycle in $G$, $j-i\leq t$ and $g+i-j\leq t$. If
$t\geq3$, then the set
$$W=V(G)\setminus\{x_2,v_1,v_2,\ldots,v_{i-1},v_{i+2},\ldots, v_g\}$$
is not a resolving set for $G$, because $|W|=\beta(G)-1$. Therefore, there exist
vertices $a,b\in V(G)\setminus W$ such that $r(a|W)=r(b|W)$.
Since $\{v_i,v_{i+1}\}$ is a basis
for $C_g$ and the distances in $C_g$ and $G$ are the same, $W\cap V(C_g)$ resolves $C_g$.
Consequently, $x_2\in\{a,b\}$, say $x_2=a$. Then, $b\in V(C_g)\setminus\{v_i,v_{i+1}\}$.
Hence, $d(b,x_1)=d(a,x_1)=1$, because $x_1\in W$. Note that $x_1$ is an internal vertex
of the ear $G_1$  of the graph $C_g\cup G_1$ and hence $x_1$ is of degree $2$ in
 $C_g\cup G_1$. But $x_1$ is adjacent to vertices $x_0,x_2$ and $b$.
Since $b\notin\{x_0,x_2\}$, $x_1$ is of degree at least $3$ in $C_g\cup G_1$.
This contradiction
implies that $r\leq2$. Therefore, $j-i\leq 2$ and $g+i-j\leq 2$.
Thus, $g\leq 4$. If $g=3$, then $\beta(G)=n-3+2=n-1$ and by Theorem~\ref{thm:B=n-1}, $G=K_n$.
If $g=4$, then $\beta(G)=n-4+2=n-2$ and by Theorem~\ref{thm:n-2}, $G$ is  $K_{r,s}$,
$K_r\vee {\overline K_s}$, or $K_r\vee(K_s\cup K_1)$. But $K_{r,s},~r,s\geq2$ is
the only graph among these graphs whose girth is $4$.
}\end{proof}

\end{document}